\documentclass[10pt,reqno]{amsart}
\usepackage{amsmath}
\usepackage{amssymb}
\usepackage{amsthm}

\newlength{\defbaselineskip}
\newcommand{\setlinespacing}[1]%
           {\setlength{\baselineskip}{#1 \defbaselineskip}}

\numberwithin{equation}{section}

\newtheorem{thm}{Theorem}[section]
\newtheorem{cor}[thm]{Corollary}
\newtheorem{lem}[thm]{Lemma}

\theoremstyle{definition}

\theoremstyle{remark}

\numberwithin{equation}{section}

\begin{document}

\title[Weighted $L^2$ estimates for the wave equation]
{On weighted $L^2$ estimates for solutions of the wave equation}

\author{Youngwoo Koh and Ihyeok Seo}

\subjclass[2010]{Primary: 35B45; Secondary: 35L05, 42B35}
\keywords{Weighted estimates, Wave equation, Morrey-Campanato class}
\thanks{The first named author was supported by NRF grant 2012-008373.}

\address{School of Mathematics, Korea Institute for Advanced Study, Seoul 130-722, Republic of Korea}
\email{ywkoh@kias.re.kr}

\address{Department of Mathematics, Sungkyunkwan University, Suwon 440-746, Republic of Korea}
\email{ihseo@skku.edu}

\maketitle


\begin{abstract}
In this paper we consider weighted $L^2$ integrability for solutions of the wave equation.
For this, we obtain some weighed $L^2$ estimates for the solutions with weights in Morrey-Campanato classes.
Our method is based on a combination of bilinear interpolation and localization argument which makes use of
Littlewood-Paley theorem and a property of Hardy-Littlewood maximal functions.
We also apply the estimates to the problem of well-posedness for wave equations with potentials.
\end{abstract}

\section{Introduction}

Let us first consider the following Cauchy problem associated with the wave equation
in $\mathbb{R}^{n+1}$, $n\geq2$:
\begin{equation}\label{wave}
\begin{cases}
-\partial_t^2u+\Delta u=F(x,t),\\
\,\qquad u(x,0)=f(x),\\
\text{ }\quad\partial_t u(x,0)=g(x).
\end{cases}
\end{equation}
Then it is well known that the solution is given by
\begin{equation}\label{sol}
u(x,t)=\cos(t\sqrt{-\Delta})f+\frac{\sin(t\sqrt{-\Delta})}{\sqrt{-\Delta}}g
-\int_0^t\frac{\sin((t-s)\sqrt{-\Delta})}{\sqrt{-\Delta}}F(s)ds
\end{equation}
in view of the Fourier transform.

The aim of this paper is to find a suitable relation between the Cauchy data $f,g$, the forcing term $F$
and the weight $w(x,t)$ which guarantee that the solution lies in weighted $L^2$ spaces, $L_{x,t}^2(w)$.

A natural way of approaching this problem may be to control weighted $L^2$ integrability
of the solution in terms of regularity of $f,g,F$.
In fact our basic strategy is to obtain the following type of estimates:
\begin{equation}\label{wei}
\|u\|_{L_{x,t}^2(w)}\leq C_{s,\widetilde{s},q,r}(w)\big(\|f\|_{\dot{H}^s}+\|g\|_{\dot{H}^{s-1}}+\|F\|_{L^q_t\dot B_{r,2}^{\widetilde{s}}}\big),
\end{equation}
where $\dot B_{r,2}^{\widetilde{s}}$ is the usual homogeneous Besov space ({\it cf. \cite{BL}}),
and $C_{s,\widetilde{s},q,r}(w)$ is a suitable constant depending on
the regularity indexes $s,\widetilde{s},q,r$ and the weight $w$.
The point here is that $C_{s,\widetilde{s},q,r}(w)$ reflects some information about the relation
between the weight and the regularity of $f,g,F$.
By considering the operator $e^{it\sqrt{-\Delta}}$, the estimate \eqref{wei} will consist of
the homogeneous estimate
\begin{equation*}
\big\|e^{it\sqrt{-\Delta}}f\big\|_{L_{x,t}^2(w)}\leq C_s(w)\|f\|_{\dot{H}^s}
\end{equation*}
and the inhomogeneous estimate
\begin{equation*}
\bigg\|\int_0^te^{i(t-s)\sqrt{-\Delta}}F(\cdot,s)ds\bigg\|_{L^2_{x,t}(w)}
\leq C_{\widetilde{s},q,r}(w)\|F\|_{L^q_t\dot B_{r,2}^{\widetilde{s}+1}}.
\end{equation*}

Before stating our results, we need to introduce a function class.
For $\alpha>0$ and $1\leq p\leq n/\alpha$,
a function $f\in L_{\textrm{loc}}^p(\mathbb{R}^{n})$ is said to be in
the Morrey-Campanato class $\mathfrak{L}^{\alpha,p}$ if
$$\|f\|_{\mathfrak{L}^{\alpha,p}}
:=\sup_{Q\text{ cubes in }\mathbb{R}^{n}}|Q|^{\alpha/n} \Big( \frac{1}{|Q|} \int_{Q} |f(y)|^p dy \Big)^{\frac{1}{p}}<\infty.$$
In particular, $\mathfrak{L}^{\alpha,p}=L^{p}$ when $p=n/\alpha$,
and even $L^{n/\alpha,\infty}\subset\mathfrak{L}^{\alpha,p}$ for $p<n/\alpha$.

Then we have the following result which can be regarded as the weighted $L^2$ Strichartz estimates
for the wave equation.
Strichartz estimates on weighted $L^2$ spaces have been studied
for the Schr\"odinger equation (\cite{RV,BBRV,BBCRV,S}).

\begin{thm}\label{thm}
Let $n \geq 2$.
Then we have for $(n+1)/4\leq s<n/2$ and $1<p\leq(n+1)/(2s+1)$
\begin{equation}\label{hop}
\big\|e^{it\sqrt{-\Delta}}f\big\|_{L^2(w(x,t))}\leq C\|w\|_{\mathfrak{L}^{2s+1,p}}^{1/2}\|f\|_{\dot{H}^s},
\end{equation}
and for $\widetilde{s}>(n-1)/2$ and $1\leq q,r\leq2$
\begin{equation}\label{inho}
\bigg\|\int_{0}^{t}e^{i(t-s)\sqrt{-\Delta}}F(s)ds\bigg\|_{L^2(w(x,t))}
\leq C\|w\|_{\mathfrak{L}^{\alpha,p}}^{1/2}\|F\|_{L^q_t\dot B_{r,2}^{\widetilde{s}+1}}
\end{equation}
with $1<p\leq(n+1)/\alpha$ and
\begin{equation}\label{sca}
\alpha=2(\widetilde{s}-\frac1q-\frac nr+\frac{n+4}2)+1.
\end{equation}
\end{thm}

\noindent\textit{Remarks.} \textit{(i)}
From the definition it is clear that
$\|f(\lambda\cdot)\|_{\mathfrak{L}^{\alpha,p}}=\lambda^{-\alpha}\|f\|_{\mathfrak{L}^{\alpha,p}}$.
Using this one can see that $\alpha=2s+1$ is the only possible index which allows \eqref{hop}
to be invariant under the scaling $(x,t)\rightarrow(\lambda x,\lambda t)$, $\lambda>0$.

\smallskip

\noindent\textit{(ii)}
In the special case where $\widetilde{s}+1=0$,
one can see that $L^r\subset\dot B_{r,2}^{\widetilde{s}+1}$ for $1<r\leq2$,
and \eqref{sca} is just the scaling condition for \eqref{inho}
with $L^r$ instead of $\dot B_{r,2}^{\widetilde{s}+1}$.

\smallskip

\noindent\textit{(iii)}
From the classical Strichartz's estimate (\cite{Str}), one can see that
\begin{equation*}
\big\|e^{it\sqrt{-\triangle}}f\big\|_{L^r_{x,t}}
\leq C\|f\|_{\dot{H}^s}
\end{equation*}
for $2(n+1)/(n-1)\leq r<\infty$ and $s=n/2-(n+1)/r$.
Then, by this and H\"older's inequality,
it follows that
for $1/2\leq s<n/2$,
\begin{equation}\label{fcs}
\big\|e^{it\sqrt{-\triangle}}f\big\|_{L^2(w(x,t))}^2
\leq\|w\|_{L^{\frac{r}{r-2}}}\|e^{it\sqrt{-\triangle}}f\|_{L^r}^2
\leq C\|w\|_{L^{\frac{n+1}{2s+1}}}\|f\|_{\dot{H}^s}^2.
\end{equation}
Hence our estimate \eqref{hop} can be seen as natural extensions to
the Morrey-Campanato classes of \eqref{fcs}.

\medskip

Some weighted $L_{x,t}^2$ estimates are known for the wave equation.
In particular, \eqref{hop} can be found in \cite{RV} for $w(x,t)$ satisfying $\sup_tw(x,t)\in\mathfrak{L}^{2,p}(\mathbb{R}^n)$
with $p>(n-1)/2$, $n\geq3$.
In fact, it is easy to see that this condition implies $w\in\mathfrak{L}^{2,p}(\mathbb{R}^{n+1})$ for the same $p$.
For a specific weight $w(x,t)=|x|^{-(2s+1)}$, $0<s<(n-1)/2$,
\eqref{hop} is proved in \cite{HMSSZ} (see (3.6) there).
Note that $|x|^{-(2s+1)}\in\mathfrak{L}^{2s+1,p}(\mathbb{R}^{n+1})$ for $p<n/(2s+1)$.
But, our theorem gives estimates for more general time-dependent weights $w(x,t)$.
Also, the following smoothing estimates (known for Morawetz estimates)
\begin{equation}\label{mor}
\big\||x|^{-b/2}e^{itD^a}f\big\|_{L_{t,x}^2}\leq C\|D^{(b-a)/2}f\|_2
\end{equation}
have been proved by many authors
for the wave equation ($a=1$) \cite{M} and for the Schr\"odinger equation ($a=2$) \cite{KY,S,V}.
For more general Morawetz estimates with angular smoothing, see \cite{H,S,FW}.
These estimates can be compared with \eqref{hop} for the wave equation ($a=1$).
Indeed, since $w(x)$ is in $\mathfrak{L}^{\alpha,p}(\mathbb{R}^{n+1})$
if $w(x)\in\mathfrak{L}^{\alpha,p}(\mathbb{R}^n)$ for $1\leq p\leq n/\alpha$,
the following corollary is directly deduced from \eqref{hop} with $s=(b-1)/2$.
Then, \eqref{mor} with $a=1$ and $(n+3)/2\leq b<n$ is just a particular case of \eqref{coo}
because $|x|^{-b}\in\mathfrak{L}^{b,p}(\mathbb{R}^n)$.

\begin{cor}
Let $n>3$. Then we have
\begin{equation}\label{coo}
\big\||w(x)|^{1/2}e^{it\sqrt{-\triangle}}f\big\|_{L_{t,x}^2}
\leq C\|D^{(b-1)/2}f\|_2
\end{equation}
if $w\in\mathfrak{L}^{b,p}(\mathbb{R}^n)$ for $(n+3)/2\leq b<n$ and $1<p\leq n/b$.
\end{cor}

Compared with the index $\alpha=2s+1$ in \eqref{hop},
we can have the same index of $\alpha$ in \eqref{inho} if
$$\frac{n+1}{4}\leq\widetilde{s}-\frac1q-\frac{n}r+\frac{n+4}{2}<\frac n2$$
(see \eqref{sca}).
Hence, setting $s=\widetilde{s}-1/q-n/r+(n+4)/2$ and using the fact that
$$\|(\sqrt{-\Delta})^{-1}f\|_{\dot B_{r,2}^{\widetilde{s}+1}}\leq C\|f\|_{\dot B_{r,2}^{\widetilde{s}}},$$
the following corollary is deduced from a simple combination of
\eqref{sol}, \eqref{hop} and \eqref{inho}.

\begin{cor}
Let $n\geq 2$.
If $u$ is a solution of the Cauchy problem \eqref{wave}, then
\begin{equation*}
\|u\|_{L^2(w(x,t))}\leq C\|w\|_{\mathfrak{L}^{2s+1,p}}^{1/2}\Big(\|f\|_{\dot{H}^s}+\|g\|_{\dot{H}^{s-1}}+
\|F\|_{L^q_t\dot B_{r,2}^{\widetilde{s}}}\Big)
\end{equation*}
if $\frac{n+1}{4}\leq s<\frac{n}{2}$, $1< p\leq\frac{n+1}{2s+1}$,
$\widetilde{s}>\frac{n-1}2$, and $1\leq q,r\leq2$, with
$s=\widetilde{s}-\frac1q-\frac{n}r+\frac{n+4}{2}$.
\end{cor}

Let us sketch the organization of the paper. In Section \ref{sec2}, we obtain a property
of the Morrey-Campanato class
regarding the Hardy-Littlewood maximal function, which is to be used for the proof of the weighted $L^2$ estimates in Theorem \ref{thm}.
Then, using bilinear interpolation and a localization argument based on Littlewood-Paley theorem in weighted $L^2$ spaces,
we prove Theorem \ref{thm} in Section \ref{sec3}.
Finally in Section \ref{sec4}, we apply the estimates to the well-posedness theory for wave equations
with potentials.

Throughout this paper, we will use the letter $C$ to denote a constant which may be different
at each occurrence. We also denote by $\widehat{f}$ the Fourier transform of $f$
and by $\langle f,g\rangle$ the usual inner product of $f,g$ on $L^2$.
Given two complex Banach spaces $A_0$ and $A_1$,
we denote by $(A_0,A_1)_{\theta,q}$
the real interpolation spaces for $0<\theta<1$ and $1\leq q\leq\infty$.
In particular, $(A_0,A_1)_{\theta,q}=A_0=A_1$ if $A_0=A_1$.
See ~\cite{BL,T} for details.


\section{Preliminary lemmas}\label{sec2}

In this section we present some preliminary lemmas which will be used for the proof of Theorem \ref{thm}.

A weight $w:\mathbb{R}^n\rightarrow[0,\infty]$ is a locally integrable function
that is allowed to be zero or infinite only on a set of Lebesgue measure zero,
and we denote by $w\in A_2(\mathbb{R}^n)$ to mean that
$w$ is in the Muckenhoupt $A_2(\mathbb{R}^n)$ class
which is defined by
\begin{equation*}
\sup_{Q\text{ cubes in }\mathbb{R}^{n}}
\bigg(\frac1{|Q|}\int_{Q}w(x)dx\bigg)\bigg(\frac1{|Q|}\int_{Q}w(x)^{-1}dx\bigg)<C_{A_2}.
\end{equation*}
Also, $w$ is said to be in the class $A_1$ if there is a constant $C_{A_1}$
such that for almost every $x$
\begin{equation*}
M(w)(x)\leq C_{A_1}w(x),
\end{equation*}
where $M(w)$ is the Hardy-Littlewood maximal function of $w$ given by
$$M(w)(x)=\sup_{Q}\frac1{|Q|}\int_{Q}w(y)dy,$$
where the sup is taken over all cubes $Q$ in $\mathbb{R}^{n}$ with center $x$.
Then,
\begin{equation}\label{bac}
A_1\subset A_2\quad\text{with}\quad C_{A_2}\leq C_{A_1}.
\end{equation}
See, for example, \cite{G} for more details.
Also, the following lemma can be found in Chapter 5 of \cite{St}.
(See also Proposition 2 in \cite{CR}.)

\begin{lem}\label{lem0}
If $M(w)(x)<\infty$ for almost every $x\in\mathbb{R}^n$, then for $0<\delta<1$
\begin{equation*}
(M(w))^\delta\in A_1
\end{equation*}
with $C_{A_1}$ independent of $w$.
\end{lem}

Next we obtain the following property of the Morrey-Campanato class
regarding the Hardy-Littlewood maximal function.
A similar property when $\alpha=2$ can be also found in \cite{CS}.

\begin{lem}\label{lem2}
Let $w\in\mathfrak{L}^{\alpha,p}$ be a weight on $\mathbb{R}^{n+1}$,
and let $w_*(x,t)$ be the $n$-dimensional maximal function defined by
$$w_*(x,t)=\sup_{Q'}\Big(\frac{1}{|Q'|}\int_{Q'}w(y,t)^\rho dy\Big)^{\frac{1}{\rho}},\quad\rho>1,$$
where $Q'$ denotes a cube in $\mathbb{R}^n$ with center $x$.
Then, if $p>\rho$ and $p>1/\alpha$, we have
$$\sup_{Q}|Q|^{\frac{\alpha}{n+1}} \Big( \frac{1}{|Q|} \int_Q w_*(x,t)^p dxdt \Big)^{\frac{1}{p}}
\leq C \sup_{Q} |Q|^{\frac{\alpha}{n+1}} \Big( \frac{1}{|Q|} \int_Q w (y,t)^p dydt \Big)^{\frac{1}{p}}.$$
Namely, if $p>\rho$ and $p>1/\alpha$,
$\|w_*\|_{\mathfrak{L}^{\alpha,p}}\leq C\|w\|_{\mathfrak{L}^{\alpha,p}}$.
Furthermore, $w_\ast(\cdot,t)\in A_2(\mathbb{R}^n)$ in the $x$ variable with a constant $C_{A_2}$ uniform in almost every $t\in\mathbb{R}$.
\end{lem}

\begin{proof}
Fix a cube $Q$ in $\mathbb{R}^{n+1}$ with center at $(z,\tau)$ and side length $\delta$.
Let us define the rectangles $R_k$, $k \geq 1$, such that $(y,t) \in R_k$ if $|t-\tau |< 2\delta$
and $y \in \widetilde{Q}(z,2^{k+1} \delta)\setminus \widetilde{Q}(z, 2^k \delta)$.
Here, $\widetilde{Q}(z,r)$ denote a cube in $\mathbb{R}^{n}$ with center $z$ and side length $r$.
Also, let $R_0 = 4Q$.

Now, setting $w^{(k)}= w \chi_{R_k}$ with the characteristic function $\chi_{R_k}$ of the set $R_k$,
we may write
$$w(y,t)=\sum_{k\geq 0}w^{(k)}(y,t)+\phi(y,t)$$
where $\phi(y,t)$ is supported on $\mathbb{R}^{n+1}\setminus\bigcup_{k\geq0}R_k$.
Then it is easy to see that
$$w_*(x,t)\leq\sum_{k \geq 0}w^{(k)}_*(x,t)+\phi_*(x,t)$$
and
$$\Big( \frac{1}{|Q|}\int_Q w_*(x,t)^pdxdt\Big)^{\frac{1}{p}}
\leq \sum_{k \geq 0}\Big(\frac{1}{|Q|}\int_Qw^{(k)}_*(x,t)^p dxdt \Big)^{\frac{1}{p}}
+\Big( \frac{1}{|Q|}\int_Q\phi_* (x,t)^p dxdt \Big)^{\frac{1}{p}}.$$
Since $(x,t)\in Q$, it is clear that $\phi_*(x,t)=0$.
Also, applying the well-known maximal theorem, $\|M(f)\|_q\leq C\|f\|_q$, $q>1$, with $q=p/\rho$,
we see that if $p>\rho$
\begin{align}\label{sso}
\nonumber|Q|^{\frac{\alpha}{n+1}} \Big( \frac{1}{|Q|} \int_Q w^{(0)}_* (x,t)^p dxdt \Big)^{\frac{1}{p}}
&\leq C |Q|^{\frac{\alpha}{n+1}} \Big( \frac{1}{|Q|} \int_{4Q} w (y,t)^p dydt \Big)^{\frac{1}{p}}\\
&\leq C\|w\|_{\mathfrak{L}^{\alpha,p}}.
\end{align}
Consequently, we only need to consider the case where $k\geq1$.

Let $k \geq 1$.
Since $(x,t)\in Q$, it follows that
\begin{align*}
w^{(k)}_*(x,t)
&= \sup_{Q'}\bigg(\frac{1}{|Q'|}\int_{Q'}w(y,t)^\rho\chi_{R_k}(y,t)dy\bigg)^{\frac{1}{\rho}}\\
&\leq C\bigg(\frac{1}{(2^k \delta)^n}\int_{\widetilde{Q}(z,2^{k+1} \delta)\setminus \widetilde{Q}(z, 2^k \delta)} w(y,t)^\rho dy \bigg)^{\frac{1}{\rho}}\\
&\leq C\bigg(\frac{1}{(2^k \delta)^n}
\int_{\widetilde{Q}(z,2^{k+1}\delta)\setminus \widetilde{Q}(z, 2^k \delta)} w(y,t)^pdy\bigg)^{\frac{1}{p}},
\end{align*}
where we used H\"older's inequality for the last inequality since $p\geq\rho$.
Thus,
\begin{align*}
\int_{Q} w^{(k)}_*(x,t)^pdxdt&\leq \frac{C}{(2^k \delta)^n}\int_{|\tau-t|<\delta}
\int_{\widetilde{Q}(z,2^{k+1}\delta)\setminus \widetilde{Q}(z, 2^k \delta)}w(y,t)^p\int_{|z-x|<\delta}1\,dxdydt\\
&\leq\frac{C}{2^{kn}}\int_{R_k}w(y,t)^pdydt.
\end{align*}
Since we can clearly choose a cube $Q_k$ in $\mathbb{R}^{n+1}$
such that $R_k\subset Q_k$ and $|Q_k| \sim (2^k \delta)^{n+1}$, we now get
\begin{align*}
|Q|^{\frac{\alpha}{n+1}} \Big( \frac{1}{|Q|} \int_{Q} w^{(k)}_* (x,t)^p dxdt \Big)^{\frac{1}{p}}
&\leq C|Q|^{\frac{\alpha}{n+1}} \Big( \frac{1}{2^{kn}|Q|}\int_{R_k}w(y,t)^pdydt \Big)^{\frac{1}{p}} \\
&\leq C |Q|^{\frac{\alpha}{n+1}} \Big( \frac{2^k}{|Q_k|} \int_{Q_k} w(y,t)^p dy dt \Big)^{\frac{1}{p}} \\
&\leq C 2^{-\alpha k + \frac{k}{p}} |Q_k|^{\frac{\alpha}{n+1}} \Big( \frac{1}{|Q_k|} \int_{Q_k} w(y,t)^p dy dt \Big)^{\frac{1}{p}}.
\end{align*}
Hence, if $p>1/\alpha$ and $p\geq\rho$, we conclude that
 $$\sum_{k \geq 1}|Q|^{\frac{\alpha}{n+1}}\Big(\frac{1}{|Q|}\int_Qw^{(k)}_*(x,t)^p dxdt \Big)^{\frac{1}{p}}
 \leq C\|w\|_{\mathfrak{L}^{\alpha,p}}.$$
By combining this and \eqref{sso}, if $p>\rho$ and $p>1/\alpha$, we get
$\|w_*\|_{\mathfrak{L}^{\alpha,p}}\leq C\|w\|_{\mathfrak{L}^{\alpha,p}}$
as desired.

Finally, to show the last assertion in the lemma, note first that
$$w_\ast(x,t)=(M(w(\cdot,t)^\rho))^{1/\rho}.$$
Since $w\in\mathfrak{L}^{\alpha,p}$ and $p\geq\rho$, it is not difficult to see that
$M(w(\cdot,t)^\rho)<\infty$ for almost every $x\in\mathbb{R}^n$.
Now, applying Lemma \ref{lem0} with $\delta=1/\rho$,
we conclude that $w_\ast(\cdot,t)\in A_1$ with $C_{A_1}$ uniform in $t\in\mathbb{R}$,
which in turn implies that
$w_\ast(\cdot,t)\in A_2$ with $C_{A_2}$ uniform in $t\in\mathbb{R}$ (see \eqref{bac}).
\end{proof}

\section{Proof of Theorem \ref{thm}}\label{sec3}
Since $w\leq w_\ast$ and
$\|w_*\|_{\mathfrak{L}^{\alpha,p}}\leq C\|w\|_{\mathfrak{L}^{\alpha,p}}$ for $p>\rho>1$ and $p>1/\alpha$
(see Lemma \ref{lem2}),
it is enough to prove Theorem \ref{thm} replacing $w$ with $w_\ast$.
The motivation behind this replacement is that
$w_\ast(\cdot,t)\in A_2(\mathbb{R}^n)$ in the $x$ variable with a constant $C_{A_2}$ uniform in almost every $t\in\mathbb{R}$ (see Lemma \ref{lem2}).
This $A_2$ condition will enable us to use a localization argument in weighted $L^2$ spaces.
Consequently, we may assume that $w$ satisfies the same $A_2$ condition,
in proving the theorem.

\subsection{Homogeneous estimates}
First we prove the homogeneous estimate \eqref{hop}.
Let $\phi$ be a smooth function supported in $(1/2,2)$ such that
$$\sum_{k=-\infty}^\infty \phi(2^k t)=1,\quad t>0.$$
For $k\in\mathbb{Z}$, we define the multiplier operators $P_kf$ by
$$\widehat{P_kf}(\xi)=\phi(2^{-k}|\xi|)\widehat{f}(\xi).$$
First we claim that if $(n+1)/4\leq s<n/2$ and $1<p\leq(n+1)/(2s+1)$
\begin{equation}\label{freq}
\big\|e^{it\sqrt{-\Delta}}P_0f\big\|_{L^2(w(x,t))}
\leq C\|w\|_{\mathfrak{L}^{2s+1,p}}^{1/2}\|f\|_2.
\end{equation}
Then it follows from scaling that
\begin{align*}
\big\|e^{it\sqrt{-\Delta}}P_k f\big\|_{L^2(w(x,t))}^2
&\leq C2^{-kn}2^{-k}\big\|e^{it\sqrt{-\Delta}}P_0(f(2^{-k}\cdot))\big\|_{L^2(w(2^{-k}x,2^{-k}t))}^2\\
&\leq C2^{-kn}2^{-k}\|w(2^{-k}x,2^{-k}t)\|_{\mathfrak{L}^{2s+1,p}}\|f(2^{-k}\cdot)\|_2^2\\
&\leq C2^{2ks}\|w\|_{\mathfrak{L}^{2s+1,p}}\|f\|_2^2.
\end{align*}
Since $w(\cdot,t)\in A_2(\mathbb{R}^n)$ uniformly for almost every $t \in \mathbb{R}$,
by the Littlewood-Paley theorem on weighted $L^2$ spaces (see Theorem 1 in \cite{Ku}), we see that
\begin{align*}
\big\|e^{it\sqrt{-\Delta}}f\big\|_{L^2(w(x,t))}^2
&=\int\big\|e^{it\sqrt{-\Delta}}f\big\|_{L^2(w(\cdot,t))}^2dt\\
&\leq C\int\bigg\|\bigg(\sum_k\big|P_ke^{it\sqrt{-\Delta}}f\big|^2\bigg)^{1/2}\bigg\|_{L^2(w(\cdot,t))}^2dt\\
&=C\sum_k\big\|e^{it\sqrt{-\Delta}}P_kf\big\|_{L^2(w(x,t))}^2.
\end{align*}
On the other hand, since $P_kP_jf=0$ if $|j-k|\geq2$, it follows that
\begin{align*}
\sum_k\big\|e^{it\sqrt{-\Delta}}P_kf\big\|_{L^2(w(x,t))}^2
&=\sum_k\big\|e^{it\sqrt{-\Delta}}P_k\big(\sum_{|j-k|\leq1}P_jf\big)\big\|_{L^2(w(x,t))}^2\\
&\leq C\|w\|_{\mathfrak{L}^{2s+1,p}}\sum_k 2^{2ks}\big\|\sum_{|j-k|\leq1}P_jf\big\|_2^2\\
&\leq C\|w\|_{\mathfrak{L}^{2s+1,p}}\|f\|_{\dot{H}^s}^2.
\end{align*}
Consequently, we get the desired estimate
$$\big\|e^{it\sqrt{-\Delta}}f\big\|_{L^2(w(x,t))}
\leq C\|w\|_{\mathfrak{L}^{2s+1,p}}^{1/2}\|f\|_{\dot{H}^s}$$
for $(n+1)/4\leq s<n/2$ and $1<p\leq(n+1)/(2s+1)$.

Now it remains to show \eqref{freq}.
By duality \eqref{freq} is equivalent to
\begin{equation*}
\bigg\|\int e^{-is\sqrt{-\Delta}}P_0F(\cdot,s)ds\bigg\|_{L_x^2}
\leq C\|w\|_{\mathfrak{L}^{2s+1,p}}^{1/2}\|F\|_{L^2(w^{-1})}.
\end{equation*}
Hence it suffices to show the following bilinear form estimate
\begin{equation*}
\bigg|\bigg\langle\int_{\mathbb{R}}e^{i(t-s)\sqrt{-\Delta}}P_0^2F(\cdot,s)ds,G(x,t)\bigg\rangle\bigg|
\leq C\|w\|_{\mathfrak{L}^{2s+1,p}}\|F\|_{L^2(w^{-1})}\|G\|_{L^2(w^{-1})}.
\end{equation*}
To show this, we first write
$$\int_{\mathbb{R}}e^{i(t-s)\sqrt{-\Delta}}P_0^2F(\cdot,s)ds=K\ast F,$$
where
$$K (x,t)=\int_{\mathbb{R}^n}e^{i(x\cdot\xi+t|\xi|)}\phi(|\xi|)^2d\xi.$$
Next, we decompose the kernel $K$ in the following way
$$\big| \big\langle K \ast F, G \big\rangle \big|
\leq \sum_{j\geq0} \big| \big\langle (\psi_j K)\ast F,G\big\rangle \big|,$$
where $\psi_j:\mathbb{R}^{n+1}\rightarrow[0,1]$ is a smooth function which is supported in $B(0,2^j) \setminus B(0,2^{j-2})$ for $j\geq1$ and in $B(0,1)$ for $j=0$,
such that $\sum_{j\geq0}\psi_j=1$.
Then it suffices to show that
\begin{equation}\label{345}
\sum_{j\geq0}\big| \big\langle (\psi_j K)\ast F,G\big\rangle \big|
\leq C\|w\|_{\mathfrak{L}^{2s+1,p}}\|F\|_{L^2(w^{-1})}\|G\|_{L^2(w^{-1})}.
\end{equation}
For this, we assume for the moment that
\begin{align}\label{21}
\big| \big\langle (\psi_j K) \ast F, G \big\rangle \big|
&\leq C2^{j (\frac{n+3}{2} -(2s+1)p)}\|w\|_{\mathfrak{L}^{2s+1,p}}^p \|F\|_{L^2(w^{-p})}\|G\|_{L^2(w^{-p})},\\
\label{22}\big| \big\langle (\psi_j K) \ast F, G \big\rangle \big|
&\leq C2^{j (\frac{n+3}{2} -\frac{2s+1}{2}p)}\|w\|_{\mathfrak{L}^{2s+1,p}}^{p/2}\|F\|_{L^2(w^{-p})}\|G\|_2,\\
\label{23}\big| \big\langle (\psi_j K) \ast F, G \big\rangle \big|
&\leq C2^{j (\frac{n+3}{2} -\frac{2s+1}{2}p)}\|w\|_{\mathfrak{L}^{2s+1,p}}^{p/2}\|F\|_2\|G\|_{L^2(w^{-p})},
\end{align}
and use the following bilinear interpolation lemma (see \cite{BL}, Section 3.13, Exercise 5(b))
as in \cite{KT,BBCRV}.

\begin{lem}\label{lem}
For $i=0,1$, let $A_i,B_i,C_i$ be Banach spaces and let $T$ be a bilinear operator such that
$T:A_0\times B_0\rightarrow C_0$,
$T:A_0\times B_1\rightarrow C_1$, and
$T:A_1\times B_0\rightarrow C_1$.
Then one has for $\theta=\theta_0+\theta_1$ and $1/q+1/r\geq1$
$$T:(A_0,A_1)_{\theta_0,q}\times(B_0,B_1)_{\theta_1,r}\rightarrow(C_0,C_1)_{\theta,1}.$$
Here\, $0<\theta_i<\theta<1$ and $1\leq q,r\leq\infty$.
\end{lem}

Indeed, let us first define the bilinear vector-valued operator $T$ by
\begin{equation}\label{55}
T(F,G)=\big\{\big\langle (\psi_j K)\ast F,G\big\rangle\big\}_{j\geq0}.
\end{equation}
Then \eqref{345} is equivalent to
\begin{equation}\label{bii}
T:L^2(w^{-1})\times L^2(w^{-1})\rightarrow\ell_1^0(\mathbb{C})
\end{equation}
with the operator norm  $C\|w\|_{\mathfrak{L}^{2s+1,p}}$.
Here, for $a\in\mathbb{R}$ and $1\leq p\leq\infty$,
$\ell^a_p(\mathbb{C})$ denotes the weighted sequence space with the norm
$$\|\{x_j\}_{j\geq0} \|_{\ell^a_p} =
\begin{cases}
\big(\sum_{j\geq0}2^{jap}|x_j|^p\big)^{\frac{1}{p}},
\quad\text{if}\quad p\neq\infty,\\
\,\sup_{j\geq0}2^{ja}|x_j|,
\quad\text{if}\quad p=\infty.
\end{cases}$$
Note that the above three estimates \eqref{21}, \eqref{22} and \eqref{23} become
\begin{align}
\nonumber\|T(F,G)\|_{\ell_\infty^{\beta_0}(\mathbb{C})}
&\leq C\|w\|_{\mathfrak{L}^{2s+1,p}}^p \|F\|_{L^2(w^{-p})}\|G\|_{L^2(w^{-p})},\\
\label{222}\|T(F,G)\|_{\ell_\infty^{\beta_1}(\mathbb{C})}
&\leq C\|w\|_{\mathfrak{L}^{2s+1,p}}^{p/2}\|F\|_{L^2(w^{-p})}\|G\|_2,\\
\label{233}\|T(F,G)\|_{\ell_\infty^{\beta_1}(\mathbb{C})}
&\leq C\|w\|_{\mathfrak{L}^{2s+1,p}}^{p/2}\|F\|_2\|G\|_{L^2(w^{-p})},
\end{align}
respectively, with $\beta_0=-(\frac{n+3}{2} -(2s+1)p)$
and $\beta_1=-(\frac{n+3}{2}-\frac{2s+1}{2}p)$.
Then, applying Lemma \ref{lem} with $\theta_0=\theta_1=1/p^\prime$ and $q=r=2$, we get
for $1<p<2$
$$T:(L^2(w^{-p}), L^2)_{1/p',2}\times(L^2(w^{-p}), L^2)_{1/p',2}
\rightarrow(\ell^{\beta_0}_{\infty}(\mathbb{C}),\ell^{\beta_1}_{\infty}(\mathbb{C}))_{2/p',1}$$
with the operator norm  $C\|w\|_{\mathfrak{L}^{2s+1,p}}$.
Now, we use the following real interpolation space identities (see Theorems 5.4.1 and 5.6.1 in \cite{BL}):

\begin{lem}\label{id}
Let $0<\theta<1$. Then one has
$$( L^{2} (w_0), L^{2} (w_1) )_{\theta,2} = L^2(w),\quad w= w_0^{1-\theta} w_1^{\theta},$$
and for $1\leq q_0,q_1,q\leq\infty$ and $s_0\neq s_1$,
$$(\ell^{s_0}_{q_0}, \ell^{s_1}_{q_1} )_{\theta, q}=\ell^s_q,\quad s= (1-\theta)s_0 + \theta s_1.$$
\end{lem}

Then, for $1<p<2$, we have
$$(L^2(w^{-p}), L^2)_{1/p',2}=L^2(w^{-1}),$$
and
$$(\ell^{\beta_0}_{\infty}(\mathbb{C}),\ell^{\beta_1}_{\infty}(\mathbb{C}))_{2/p',1}
=\ell_{1}^{0}(\mathbb{C})$$
if $(1-\frac2{p'})\beta_0+\frac2{p'}\beta_1=0$ (i.e., $s=(n+1)/4$).
Hence, we get \eqref{bii} if $s=(n+1)/4$ and $1<p\leq\frac{n+1}{2s+1}\,(<2)$.
When $s>(n+1)/4$, note that $\gamma:=\frac{p'}2(2s+1-\frac{n+3}2)>0$.
Since $j\geq0$ and $\beta_1<0$, the estimates \eqref{222} and \eqref{233} are trivially satisfied
for $\beta_1$ replaced by $\beta_1-\gamma$.
Hence, by the same argument we only need to check that $(1-\frac2{p'})\beta_0+\frac2{p'}(\beta_1-\gamma)=0$.
But, this is an easy computation.
Consequently we get \eqref{bii} if $(n+1)/4\leq s<n/2$ and $1<p\leq\frac{n+1}{2s+1}$,
and so the proof is completed.

\smallskip

Finally, we have to show the estimates \eqref{21}, \eqref{22} and \eqref{23}.
For $j\geq0$, let $\{Q_\lambda\}_{\lambda\in 2^j\mathbb{Z}^{n+1}}$ be a collection of
cubes $Q_\lambda\subset\mathbb{R}^{n+1}$ centered at $\lambda$ with side length $2^j$.
Then by disjointness of cubes, we see that
\begin{align*}
\big|\big\langle(\psi_j K)\ast F,G\big\rangle\big|
&\leq\sum_{\lambda,\mu\in 2^j \mathbb{Z}^{n+1}}\big|\big\langle(\psi_j K)\ast(F\chi_{Q_\lambda}),G \chi_{Q_\mu} \big\rangle\big|\\
&\leq \sum_{\lambda\in 2^j \mathbb{Z}^{n+1}} \big| \big\langle (\psi_j K)\ast
(F\chi_{Q_\lambda}),G\chi_{\widetilde{Q}_\lambda}\big\rangle\big|,
\end{align*}
where $\widetilde{Q}_\lambda$ denotes the cube with side length $2^{j+2}$ and the same center as $Q_\lambda$.
By Young's and Cauchy-Schwartz inequalities, it follows that
\begin{align}\label{33}
\big| \big\langle (\psi_j K) \ast F, G \big\rangle \big|
\nonumber&\leq \sum_{\lambda\in 2^j \mathbb{Z}^{n+1}}\|(\psi_j K)\ast(F\chi_{Q_\lambda})\|_{\infty}
\|G\chi_{\widetilde{Q}_\lambda}\|_1\\
\nonumber&\leq\sum_{\lambda\in 2^j \mathbb{Z}^{n+1}}\|\psi_jK\|_{\infty}\|F\chi_{Q_\lambda}\|_1
\|G\chi_{\widetilde{Q}_\lambda}\|_1\\
&\leq\|\psi_j K\|_{\infty}
\Big(\sum_{\lambda\in 2^j \mathbb{Z}^{n+1}}\|F\chi_{Q_\lambda}\|_1^2 \Big)^{\frac{1}{2}}
\Big(\sum_{\lambda\in2^j\mathbb{Z}^{n+1}}\|G\chi_{\widetilde{Q}_\lambda}\|_1^2\Big)^{\frac{1}{2}}.
\end{align}
Now we need to bound the terms
$$\| \psi_j K \|_{\infty},
\quad\sum_{\lambda\in 2^j \mathbb{Z}^{n+1}} \|F\chi_{Q_\lambda}\|_1^2,
\quad\sum_{\lambda\in 2^j \mathbb{Z}^{n+1}} \|G\chi_{\widetilde{Q}_\lambda}\|_1^2.$$

For the first term we use the following well-known lemma, Lemma \ref{25},
which is essentially due to Littman \cite{L}. (See also \cite{St}, VIII, Section 5, B.)
Indeed, by applying the lemma with $\psi(\xi)=|\xi|$, it follows that
$$|K(x,t)|=\bigg|\int_{\mathbb{R}^n} e^{i(x\cdot\xi + t|\xi|)} \phi(|\xi|)^2 d\xi\bigg|
\leq C(1+ |(x,t)| )^{-\frac{n-1}{2}},$$
since the Hessian matrix $H\psi$
has $n-1$ non-zero eigenvalues for each $\xi \in \{\xi\in\mathbb{R}^n: |\xi|\sim 1 \}$.
Thus we get
\begin{equation}\label{stationary_phase}
\|\psi_j K \|_{\infty}\leq C2^{-j \frac{n-1}{2} }.
\end{equation}

\begin{lem}\label{25}
Let $H \psi$ be the Hessian matrix given by $(\frac{\partial^2\psi}{\partial\xi_i\partial\xi_j})$.
Suppose that $\phi$ is a compactly supported smooth function on $\mathbb{R}^n$
and $\psi$ is a smooth function which satisfies rank $H \psi \geq k$ on the support of $\phi$.
Then, for $(x,t)\in \mathbb{R}^{n+1}$
$$\bigg|\int e^{i(x\cdot\xi+t\psi(\xi))}\phi(\xi)d\xi\bigg|
\leq C(1+|(x,t)|)^{-\frac{k}{2}}.$$
\end{lem}

Next, we have the following bound
\begin{align}\label{get_Morrey}
\sum_{\lambda\in 2^j \mathbb{Z}^{n+1}}\|F\chi_{Q_\lambda}\|_1^2
\nonumber&=\sum_{\lambda\in 2^j \mathbb{Z}^{n+1}} \Big( \int_{Q_\lambda} |F\chi_{Q_\lambda}| w^{-\frac{p}{2}} w^{\frac{p}{2}}dxdt\Big)^2\\
\nonumber&\leq\sum_{\lambda\in 2^j \mathbb{Z}^{n+1}} \Big( \int_{Q_\lambda} |F\chi_{Q_\lambda}|^2 w^{-p} dxdt \Big)\Big(\int_{Q_\lambda}w^{p}dxdt\Big)\\
\nonumber&\leq \sup_{\lambda\in 2^j \mathbb{Z}^{n+1}}\Big(\int_{Q_\lambda}w^{p}dxdt\Big)
\sum_{\lambda\in 2^j \mathbb{Z}^{n+1}} \Big( \int_{Q_\lambda} |F\chi_{Q_\lambda}|^2 w^{-p}dxdt\Big)\\
&\leq C2^{j(n+1-(2s+1)p)}\|w\|_{\mathfrak{L}^{2s+1,p}}^p \| F \|_{L^2(w^{-p})}^2
\end{align}
while
\begin{align}\label{get_trivial}
\sum_{\lambda\in 2^j\mathbb{Z}^{n+1}}\|F\chi_{Q_\lambda}\|_1^2
\nonumber&\leq\sum_{\lambda\in 2^j\mathbb{Z}^{n+1}}\|F\chi_{Q_\lambda}\|_2^2\|\chi_{Q_\lambda}\|_2^2\\
&\leq C2^{j(n+1)}\|F\|_2^2.
\end{align}
Similarly for $\sum_{\lambda\in 2^j\mathbb{Z}^{n+1}}\|G\chi_{\widetilde{Q}_\lambda}\|_1^2$.
Now, combining \eqref{33}, \eqref{stationary_phase}, \eqref{get_Morrey} and \eqref{get_trivial},
we get the desired estimates \eqref{21}, \eqref{22} and \eqref{23}.


\subsection{Inhomogeneous estimates}
Now we prove the inhomogeneous estimate \eqref{inho}.
First we claim that for $\alpha>2n+4-2n/r-2/q$ and
$1<p\leq(n+1)/\alpha$
\begin{equation}\label{base}
\bigg\|\int_{0}^{t} e^{i(t-s)\sqrt{-\Delta}} P_0 F(\cdot,s) ds \bigg\|_{L^2(w(x,t))}
\leq C\| w \|_{\mathfrak{L}^{\alpha,p}}^{1/2}\|F\|_{L^q_t L^r_x}
\end{equation}
if $1\leq r,q\leq 2$.
Then, by the Littlewood-Paley theorem on weighted $L^2$ spaces as before,
it follows that
\begin{align*}
\bigg\|\int_{0}^{t} e^{i(t-s)\sqrt{-\Delta}}F(\cdot,s)&ds\bigg\|_{L^2(w(x,t))}^2\\
&\leq C\sum_{k}\bigg\|\int_{0}^{t}e^{i(t-s)\sqrt{-\Delta}}P_k\big(\sum_{|j-k|\leq1}P_j F(\cdot,s)\big)ds \bigg\|_{L^2(w(x,t))}^2.
\end{align*}
By using \eqref{base} and scaling, the right-hand side in the above is bounded by
$$C\|w\|_{\mathfrak{L}^{\alpha,p}}
\sum_{k} 2^{k(\alpha-n-3+\frac{2n}{r}+\frac{2}{q})}\big\|\sum_{|j-k|\leq1}P_j F(\cdot,s)\big\|_{L^q_t L^r_x}^2,$$
which is in turn bounded by
$$C\|w\|_{\mathfrak{L}^{\alpha,p}}
\sum_{k} 2^{k(\alpha-n-3+\frac{2n}{r}+\frac{2}{q})}\|P_k F(\cdot,s)\|_{L^q_t L^r_x}^2.$$
Since $q\leq2$, by Minkowski's integral inequality we see that
\begin{align*}
\sum_{k}2^{k(\alpha-n-3+\frac{2n}{r}+\frac{2}{q})}\|P_kF(\cdot,t)\|_{L^q_t L^r_x}^2
&\leq\Big\|\Big(\sum_{k}2^{k(\alpha-n-3+\frac{2n}{r}+\frac{2}{q})}\|P_k F(\cdot,t)\|_{L^r_x}^2\Big)^{\frac{1}{2}} \Big\|_{L^q_t}^2 \\
&=\|F\|_{L^q_t\dot B_{r,2}^{\widetilde{s}+1}}^2
\end{align*}
with $\widetilde{s}+1=\frac12(\alpha-n-3+\frac{2n}{r}+\frac{2}{q})$.
Thus, we get the desired estimate
$$\bigg\|\int_{0}^{t} e^{i(t-s)\sqrt{-\Delta}}F(\cdot,s)ds\bigg\|_{L^2(w(x,t))}
\leq C\|w\|_{\mathfrak{L}^{\alpha,p}}^{1/2}\|F\|_{L^q_t\dot B_{r,2}^{\widetilde{s}+1}}$$
for $\alpha>2n+4-2n/r-2/q$ and $1<p\leq(n+1)/\alpha$ if $1\leq q,r\leq2$ and
\begin{equation}\label{con}
\widetilde{s}+1=\frac12(\alpha-n-3+\frac{2n}{r}+\frac{2}{q}).
\end{equation}
Note that from \eqref{con} the condition $\alpha>2n+4-2n/r-2/q$ is equivalent to
the condition $\widetilde{s}>(n-1)/2$ in Theorem \ref{thm}, and so the proof is completed.

\smallskip

Now it remains to show \eqref{base}.
We will show the following estimate
\begin{equation}\label{eno}
\bigg\|\int_{-\infty}^{t} e^{i(t-s)\sqrt{-\Delta}} P_0 F(\cdot,s) ds \bigg\|_{L^2(w(x,t))}
\leq C\|w\|_{\mathfrak{L}^{\alpha,p}}^{1/2}\|F\|_{L^q_t L^r_x}
\end{equation}
which implies \eqref{base}.
Indeed, to obtain \eqref{base} from \eqref{eno},
we first decompose the $L_t^2$ norm in the left-hand side of \eqref{base}
into two parts, $t\geq0$ and $t<0$. Then the second part can be reduced to the first one
by changing the variable $t\mapsto-t$, and so we only need to consider the first part.
But, since $[0,t)=(-\infty,t)\cap[0,\infty)$, applying \eqref{eno} with $F$ replaced by $\chi_{[0,\infty)}(s)F$,
we can bound the first part as desired.
To show \eqref{eno}, by duality it suffices to show that
$$\bigg|\bigg\langle\int_{-\infty}^{t}e^{i(t-s)\sqrt{-\Delta}}P_0F(\cdot,s)ds,G\bigg\rangle\bigg|
\leq C\|w\|_{\mathfrak{L}^{\alpha,p}}^{1/2}\|F\|_{L^q_tL^r_x}\|G\|_{L^2(w^{-1})}.$$
Let us first write
\begin{align*}
\int_{-\infty}^{t}e^{i(t-s)\sqrt{-\Delta}}P_0F(\cdot,s)ds
&=\int_{\mathbb{R}}\chi_{(0,\infty)}(t-s)e^{i(t-s)\sqrt{-\triangle}}P_0F(\cdot,s)ds\\
&=K\ast F,
\end{align*}
where
$$K(x,t)=\int_{\mathbb{R}^n}\chi_{(0,\infty)}(t)e^{i(x\cdot\xi+t|\xi|)}\phi(|\xi|)d\xi.$$
Then, with the same notations as in the previous section, it is enough to show that
\begin{equation}\label{34}
\sum_{j\geq0}\big|\big\langle(\psi_j K)\ast F,G\big\rangle\big|
\leq C\|w\|_{\mathfrak{L}^{\alpha,p}}^{1/2}\|F\|_{L^q_tL^r_x}\|G\|_{L^2(w^{-1})}.
\end{equation}
For this, we assume for the moment that
\begin{equation}\label{31}
\big|\big\langle(\psi_j K)\ast F,G\big\rangle\big|
\leq C 2^{j (\frac{n}{r'}+\frac{1}{q'}+1)}\|F\|_{L^q_tL^r_x}\|G\|_{L^2}
\end{equation}
and
\begin{equation}\label{32}
\big|\big\langle(\psi_j K)\ast F,G\big\rangle\big|
\leq C2^{j(\frac{n}{r'}+\frac{1}{q'}+1-\frac{\alpha}{2}p)}\|w\|_{\mathfrak{L}^{\alpha,p}}^{p/2}
\|F\|_{L^q_t L^r_x}\|G\|_{L^2(w^{-p})}.
\end{equation}
Then these estimates say that for $\beta_0=-(\frac{n}{r'}+\frac{1}{q'}+1)$ and $\beta_1=-(\frac{n}{r'}+\frac{1}{q'}+1-\frac{\alpha}{2}p)$
\begin{equation*}
\|T(F,G)\|_{\ell_\infty^{\beta_0}(\mathbb{C})}
\leq C\|F\|_{L^q_t L^r_x}\|G\|_{L^2}
\end{equation*}
and
\begin{equation*}
\|T(F,G)\|_{\ell_\infty^{\beta_1}(\mathbb{C})}
\leq C\|w\|_{\mathfrak{L}^{\alpha,p}}^{p/2}\|F\|_{L^q_t L^r_x}\|G\|_{L^2(w^{-p})},
\end{equation*}
where $T$ is given as in \eqref{55}.
Now, by the bilinear interpolation (see \cite{BL}, Section 3.13, Exercise 5(a))
between these two estimates, it follows that for $1<p<\infty$
$$T:(L^q_t L^r_x,L^q_t L^r_x)_{1/p,2}\times(L^2,L^2(w^{-p}))_{1/p,2}
\rightarrow(\ell^{\beta_0}_{\infty}(\mathbb{C}),\ell^{\beta_1}_{\infty}(\mathbb{C}))_{1/p,\infty}$$
with the operator norm  $C\|w\|_{\mathfrak{L}^{\alpha,p}}^{1/2}$.
Using Lemma \ref{id}, we get
\begin{equation*}
\|T(F,G)\|_{\ell_\infty^{\beta}(\mathbb{C})}
\leq C\|w\|_{\mathfrak{L}^{\alpha,p}}^{1/2}\|F\|_{L^q_tL^r_x}\|G\|_{L^2(w^{-1})}
\end{equation*}
with $\beta=(1-\frac1p)\beta_0+\frac1p\beta_1=-(\frac{n}{r'}+\frac{1}{q'}+1-\frac{\alpha}{2})$.
That is to say,
\begin{equation*}
\big|\big\langle(\psi_j K)\ast F,G\big\rangle\big|
\leq C2^{j (\frac{n}{r'}+\frac{1}{q'}+1-\frac{\alpha}{2})}\|w\|_{\mathfrak{L}^{\alpha,p}}^{1/2}
\|F\|_{L^q_tL^r_x}\|G\|_{L^2(w^{-1})} .
\end{equation*}
Thus, when $\alpha>2n+4-2n/r-2/q$, we get the desired estimate \eqref{34}.

Finally, we have to show \eqref{31} and \eqref{32}.
Recall from \eqref{33} that
$$\big|\big\langle(\psi_jK)\ast F,G\big\rangle\big|
\leq\|\psi_jK\|_{\infty}
\Big(\sum_{\lambda\in 2^j\mathbb{Z}^{n+1}}\|F\chi_{Q_\lambda}\|_1^2 \Big)^{\frac{1}{2}}
\Big(\sum_{\lambda\in 2^j\mathbb{Z}^{n+1}}\|G\chi_{\widetilde{Q}_\lambda}\|_1^2\Big)^{\frac{1}{2}}.$$
First, using Lemma \ref{25} as before, we have
\begin{equation}\label{7}
\|\psi_jK\|_{\infty}\leq C2^{-j\frac{n-1}{2}}.
\end{equation}
Next, we may write $Q_{\lambda}=Q_{\lambda(x)}\times Q_{\lambda(t)}$,
where $Q_{\lambda(x)}$ is a cube in $\mathbb{R}^n$ with respect to $x$ variable,
and $Q_{\lambda(t)}$ is an interval in $\mathbb{R}$ with respect to $t$ variable.
Then, since $q,r\leq2$, by Minkowski's integral inequality it follows that
\begin{align*}
\Big( \sum_{\lambda\in 2^j \mathbb{Z}^{n+1}} \| F\chi_{Q_\lambda} \|_1^2 \Big)^{\frac{1}{2}}
&\leq \bigg( \sum_{\lambda\in 2^j \mathbb{Z}^{n+1}} \Big( \int_{Q_{\lambda(t)}} \| F\chi_{Q_\lambda}\|_{L^r_x} |Q_{\lambda(x)}|^{\frac{1}{r'}} dt \Big)^2 \bigg)^{\frac{1}{2}} \\
&\leq \Big( \sum_{\lambda\in 2^j \mathbb{Z}^{n+1}}  \| F\chi_{Q_\lambda}\|_{L^q_t L^r_x}^2 |Q_{\lambda(x)}|^{\frac{2}{r'}} |Q_{\lambda(t)}|^{\frac{2}{q'}} \Big)^{\frac{1}{2}} \\
&\leq C2^{j (\frac{n}{r'}+\frac{1}{q'}) } \| F \|_{L^q_t L^r_x}.
\end{align*}
On the other hand, we have the following bound
\begin{align}\label{8}
\sum_{\lambda\in 2^j\mathbb{Z}^{n+1}}\|F\chi_{Q_\lambda}\|_1^2
\nonumber&=\sum_{\lambda\in 2^j\mathbb{Z}^{n+1}}\Big(\int_{Q_\lambda}|F\chi_{Q_\lambda}|w^{-\frac{p}{2}} w^{\frac{p}{2}}dxdt\Big)^2\\
\nonumber&\leq\sum_{\lambda\in2^j\mathbb{Z}^{n+1}}\Big(\int_{Q_\lambda}|F\chi_{Q_\lambda}|^2 w^{-p}dxdt\Big)
\Big(\int_{Q_\lambda}w^{p}dxdt\Big)\\
\nonumber&\leq \sup_{\lambda \in 2^j \mathbb{Z}^{n+1}} \Big(\int_{Q_\lambda}w^{p}dxdt\Big)
\sum_{\lambda\in 2^j \mathbb{Z}^{n+1}} \Big( \int_{Q_\lambda}|F\chi_{Q_\lambda}|^2w^{-p}dxdt\Big)  \\
&\leq C2^{j(n+1-\alpha p)}\|w\|_{\mathfrak{L}^{\alpha,p}}^p\|F\|_{L^2(w^{-p})}^2.
\end{align}
Combining \eqref{get_trivial}, \eqref{7} and \eqref{8}, we get the desired estimates \eqref{31} and \eqref{32}.

\section{Further applications}\label{sec4}

In this final section we consider the following Cauchy problem
for wave equations with potentials:
    \begin{equation}\label{wavepoten}
    \begin{cases}
    \partial_t^2u-\Delta u+V(x,t)u=0,\\
    u(x,0)=f(x),\\
    \partial_tu(x,0)=g(x),
    \end{cases}
    \end{equation}
where $(x,t)\in\mathbb{R}^{n+1}$, $n=2,3$.
The well-posedness for this problem in the space $L_{x,t}^2(|V|)$ was studied for $n\geq3$ in \cite{RV},
when $V=V_1+V_2$ with $V_1\in L_t^\infty\mathfrak{L}_x^{2,p}$, $(n-1)/2<p\leq n/2$, $V_2\in L_t^rL_x^\infty$, $r>1$,
and $\|V_1\|_{L_t^\infty\mathfrak{L}_x^{2,p}}$ small enough.
Here our main aim is to deal with the two-dimensional case $n=2$. Our result is the following theorem.

\begin{thm}
Let $n=2,3$.
If $n=2$ we assume that
$V\in L_t^1\mathfrak{L}_x^{1-s,r}\cap\mathfrak{L}_{x,t}^{2s+1,p}$ for $3/4\leq s<1$, $1<r\leq 2/(1-s)$ and $1<p\leq(n+1)/(2s+1)$, with
$\|V\|_{L_t^1\mathfrak{L}_x^{1-s,r}}$ and $\|V\|_{\mathfrak{L}_{x,t}^{2s+1,p}}$ small enough.
Similarly, we assume for $n=3$ the same conditions with $s=1$ and $\mathfrak{L}_x^{1-s,r}$ replaced by $L_x^{\infty}$.
Then, if $f\in\dot{H}^s$ and $g\in\dot{H}^{s-1}$, there exists a unique solution of \eqref{wavepoten}
in the space $L_{x,t}^2(|V|)$. Furthermore,
\begin{equation}\label{fur}
\|u\|_{L_{x,t}^2(|V|)}\leq C(\|f\|_{\dot{H}^s}+\|g\|_{\dot{H}^{s-1}}).
\end{equation}
\end{thm}

\begin{proof}
As a preliminary step, we recall from \cite{SW} that the fractional integral $I_\alpha$ of convolution
with $|x|^{-n+\alpha}$, $0<\alpha<n$, satisfies the following inequality
\begin{equation}\label{fpi}
\|I_\alpha f\|_{L^2(w(x))}\leq C\|w\|_{\mathfrak{L}^{\alpha,r}}^{1/2}\|f\|_{L^2},
\end{equation}
where $\alpha>0$ and $1<r\leq n/\alpha$.
Indeed, this inequality follows directly from [\cite{SW}, (1.10)] with $w=v^{-1}$ and $p=2$.
Then, using this inequality, we see that
\begin{align*}
\bigg\|\int_{0}^{t}e^{i(t-s)\sqrt{-\Delta}}|\nabla|^{-\alpha}F(\cdot,s)ds\bigg\|_{L_x^2(w(x,t))}
&\leq C\|w\|_{\mathfrak{L}_x^{\alpha,r}}^{1/2}\bigg\|\int_{0}^{t}e^{i(t-s)\sqrt{-\Delta}}F(\cdot,s)ds\bigg\|_{L_x^2}\\
\leq C&\|w\|_{\mathfrak{L}_x^{\alpha,r}}^{1/2}\bigg\|\int e^{-is\sqrt{-\Delta}}\chi_{(0,t)}(s)F(\cdot,s)ds\bigg\|_{L_x^2}
\end{align*}
because the integral kernel of the multiplier operator $|\nabla|^{-\alpha}$ is given by $|x|^{-n+\alpha}$, $0<\alpha<n$.
Combining this estimate and the following dual estimate of \eqref{hop},
\begin{equation*}
\bigg\|\int e^{-is\sqrt{-\Delta}}F(\cdot,s)ds\bigg\|_{L_x^2}
\leq C\|w\|_{\mathfrak{L}^{2s+1,p}}^{1/2}\||\nabla|^{s}F\|_{L_{x,t}^2(w^{-1})},
\end{equation*}
we get for $\alpha>0$, $1<r\leq n/\alpha$, $(n+1)/4\leq s<n/2$, and $1<p\leq(n+1)/(2s+1)$,
\begin{equation}\label{inhof3}
\bigg\|\int_{0}^{t}e^{i(t-s)\sqrt{-\Delta}}F(\cdot,s)ds\bigg\|_{L_{x,t}^2(w(x,t))}
\leq C\|w\|_{L_t^1\mathfrak{L}_x^{\alpha,r}}^{1/2}\|w\|_{\mathfrak{L}^{2s+1,p}}^{1/2}\||\nabla|^{\alpha+s}F\|_{L_{x,t}^2(w^{-1})}.
\end{equation}

Now we consider the solution of \eqref{wavepoten} which is given by
\begin{equation}\label{solf}
u(x,t)=\cos(t\sqrt{-\Delta})f+\frac{\sin(t\sqrt{-\Delta})}{\sqrt{-\Delta}}g
-\int_0^t\frac{\sin((t-s)\sqrt{-\Delta})}{\sqrt{-\Delta}}(Vu)(s)ds,
\end{equation}
where $f\in\dot{H}^s$ and $g\in\dot{H}^{s-1}$.
Then by the homogeneous estimate \eqref{hop},
we only need to show that
\begin{equation}\label{inhof34}
\bigg\|\int_{0}^{t}\frac{e^{i(t-s)\sqrt{-\Delta}}}{\sqrt{-\Delta}}(Vu)(\cdot,s)ds\bigg\|_{L_{x,t}^2(|V|)}
\leq \frac12\|u\|_{L_{x,t}^2(|V|)}
\end{equation}
in order to obtain the well-posedness using the standard fixed-point argument.
But, when $n=2$, one can use \eqref{inhof3}, with $\alpha=1-s$ and $w=|V|$, to conclude that
\begin{align*}
\bigg\|\int_{0}^{t}\frac{e^{i(t-s)\sqrt{-\Delta}}}{\sqrt{-\Delta}}(Vu)(\cdot,s)ds\bigg\|_{L_{x,t}^2(V)}
&\leq C\|V\|_{L_t^1\mathfrak{L}_x^{1-s,r}}^{1/2}\|V\|_{\mathfrak{L}^{2s+1,p}}^{1/2}\|Vu\|_{L_{x,t}^2(|V|^{-1})}\\
&\leq \frac12\|u\|_{L_{x,t}^2(|V|^{-1})}
\end{align*}
if $\|V\|_{L_t^1\mathfrak{L}_x^{1-s,r}}$ and $\|V\|_{\mathfrak{L}^{2s+1,p}}$ are sufficiently small.
Also, \eqref{fur} follows easily from combining \eqref{solf}, \eqref{hop} and \eqref{inhof34}.
Now the proof is complete for $n=2$.

When $n=3$, by repeating the above argument with $\alpha=0$ using \eqref{fpi} replaced by the trivial inequality
$\|f\|_{L^2(w(x))}\leq C\|w\|_{L^\infty}^{1/2}\|f\|_{L^2}$,
one can similarly obtain \eqref{inhof34} under the conditions given in the theorem.
We omit the details.
\end{proof}



\begin{thebibliography}{9}

\bibitem{BBCRV} J. A. Barcel\'o, J. M. Bennett, A. Carbery, A. Ruiz, M. C. Vilela,
\textit{Strichartz inequalities with weights in Morrey-Campanato classes},
Collect. Math. 61 (2010), 49-56.

\bibitem{BBRV} J. A. Barcelo, J. M. Bennett, A. Ruiz and M. C. Vilela,
\textit{Local smoothing for Kato potentials in three dimensions}, Math. Nachr. 282 (2009), 1391-1405.

\bibitem{BL} J. Bergh and J. L\"ofstr\"om, \textit{Interpolation Spaces, An Introduction},
Springer-Verlag, Berlin-New York, 1976.

\bibitem{CS} S. Chanillo and E. Sawyer, \textit{Unique continuation for $\Delta+v$ and the C. Fefferman-Phong class},
Trans. Amer. Math. Soc. 318 (1990), 275-300.

\bibitem{CR} R. Coifman and R. Rochberg, \textit{Another characterization of BMO}, Proc. Amer. Math. Soc. 79 (1980),
249-254.

\bibitem{FW} D. Fang and C. Wang, \textit{Weighted Strichartz estimates with angular regularity and their
applications}, Forum Math. 23 (2011), 181-205.

\bibitem{G} L. Grafakos, \emph{Modern Fourier Analysis}, Springer, New York, 2008.

\bibitem{HMSSZ} K. Hidano, J. Metcalfe, H. F. Smith, C. D. Sogge and Y. Zhou,
\textit{On abstract Strichartz estimates and the Strauss conjecture for nontrapping obstacles},
Trans. Amer. Math. Soc. 362 (2010), 2789-2809.

\bibitem{H} T. Hoshiro, \textit{On weighted $L^2$ estimates of solutions to wave equations},
J. Anal. Math. 72 (1997), 127-140.

\bibitem{KY} T. Kato and K. Yajima, \textit{Some examples of smooth operators and the associated
smoothing effect}, Rev. Math. Phys. 1 (1989), 481-496.

\bibitem{KT} M. Keel and T. Tao, \textit{Endpoint Strichartz estimates},
Amer. J. Math. 120 (1998), 955-980.

\bibitem{Ku} D. Kurtz, \textit{Littlewood-Peley and multiplier theorems on weighted $L^p$ spaces},
Trans. Amer. Math. Soc. 259 (1980), 235-254.

\bibitem{L} W. Littman, \textit{Fourier transforms of surface-carried measures and differentiability of
surface averages}, Bull. Amer. Math. Soc. 69 (1963), 766-770.

\bibitem{M} C. S. Morawetz, \textit{Time decay for the nonlinear Klein-Gordon equation},
Proc. Roy. Soc. A 306 (1968), 291-296.

\bibitem{RV} A. Ruiz and L. Vega, \emph{Local regularity of solutions to wave equations with time-dependent
potentials}, Duke Math. J. 76 (1994), 913-940.

\bibitem{SW} E. Sawyer and R. L. Wheeden, \emph{Weighted inequalities for fractional integrals on Euclidean and homogeneous spaces},
Amer. J. Math. 114 (1992), 813-874.

\bibitem{S} I. Seo, \textit{From resolvent estimates to unique continuation for the Schr\"odinger equation},
to appear in Trans. Amer. Math. Soc., arXiv:1401.0901.

\bibitem{St} E. M. Stein, \emph{Harmonic Analysis. Real-variable Methods, Orthogonality, and Oscillatory Integrals}, Princeton University Press, Princeton, New Jersey, 1993.

\bibitem{Str} R. S. Strichartz, \emph{Restrictions of Fourier transforms to quadratic surfaces and decay of
solutions of wave equations}, Duke Math. J. 44 (1977), 705-714.

\bibitem{S} M. Sugimoto, \textit{Global smoothing properties of generalized Schr\"odinger equations},
J. Anal. Math. 76 (1998), 191-204.

\bibitem{T} H. Triebel, \emph{Interpolation Theory, Function Spaces, Differential operator}, North-Holland,
New York, 1978.

\bibitem{V} M. C. Vilela, \textit{Regularity of solutions to the free Schr\"odinger equation with
radial initial data}, Illinois J. Math. 45 (2001), 361-370.

\end{thebibliography}
\end{document}